\documentclass[12pt,a4paper,reqno]{amsart}

\usepackage{amsmath,amsthm}
\usepackage{amssymb}

\theoremstyle{theorem}

\theoremstyle{lemma}
\newtheorem*{lemma}{Lemma}

\theoremstyle{definition}

\theoremstyle{remark}

\title{An alternative recursive formula for the sums of powers of integers}

\author[J.L. Cereceda]{Jos\'e Luis Cereceda}
\address{%
        Collado Villalba, 28400 Madrid, Spain}
\email{jl.cereceda@movistar.es}

\begin{document}

\begin{abstract}
In this note, we derive an alternative recursive formula for the sums of powers of integers involving the Stirling numbers of the first kind. As a remarkable by-product, we provide a non-recursive definition of the Catalan numbers.
\end{abstract}

\maketitle

\vspace{-.5cm}

\section{Introduction}

The sums of powers of the first $n$ positive integers $S_p(n) = 1^p + 2^p +\cdots + n^p$,
($p =0,1,2,\ldots\, $) satisfy the fundamental identity
\begin{equation}\label{id1}
1 + \sum_{t=0}^{p} \binom{p+1}{t} S_t(n) = (n+1)^{p+1}, \quad p \geq 0,
\end{equation}
from which we can successively compute $S_0(n)$, $S_1(n)$, $S_2(n)$, etc. Identity \eqref{id1} can easily be proved
by using the binomial theorem; see, e.g., \cite{kelly,owens}. Several variations of \eqref{id1} are also well known
\cite{snow,turner,acu}.

In this note, we derive the following lesser-known recursive formula for $S_p(n)$:
\begin{equation}\label{myeq}
p! + \sum_{t=0}^{p} \begin{bmatrix} p+1 \\ t+1 \end{bmatrix} S_t(n) = p! \binom{n+p+1}{p+1},
 \quad p\geq 0,
\end{equation}
where $\genfrac{[}{]}{0pt}{}{p}{t}$ denote the (unsigned) Stirling numbers of the first kind, also known as the Stirling
cycle numbers (see, e.g., \cite[Chapter 6]{temple}). Table \ref{tb:1} shows the first few rows of the Stirling number triangle.
Although the recursive formula \eqref{myeq} is by no means new, our purpose in dealing with recurrence \eqref{myeq} in this
note is two-fold. On one hand, we aim to provide a new algebraic proof of \eqref{myeq} by making use of two related identities
involving the harmonic numbers. On the other hand, as we will show, using \eqref{myeq} in conjunction with the principle of strong
mathematical induction yields the identity
\begin{table}[ttt]
\centering
\begin{tabular}{l|rrrrrrrrr}
$p \backslash  \, t  $ & 0 & 1 & 2 & 3 & 4 & 5 & 6 & 7 & 8   \\ \hline\hline
0 &  1 &  &  &  &  &  &  &  &   \\
1 &  0 & 1 &  &  &  &  &  &  &   \\
2 &  0 & 1 & 1 &  &  &  &  &  &   \\
3 &  0 & 2 & 3 & 1 &  &  &  &  &   \\
4 &  0 & 6 & 11 & 6 & 1 &  &  &  &   \\
5 &  0 & 24 & 50 & 35 & 10 & 1 &  &  &   \\
6 &  0 & 120 & 274 & 225 & 85 & 15 & 1 &  &  \\
7 &  0 & 720 & 1764 & 1624 & 735 & 175 & 21 & 1 &  \\
8 &  0 & 5040 & 13068 & 13132 & 6769 & 1960 & 322 & 28 & 1  \\ \hline
\end{tabular}\vspace{2mm}
\label{tb:1}
\caption{}
\end{table}
\begin{equation}\label{id3}
\sum_{t=0}^{p} r^t \begin{bmatrix} p+1 \\ t+1 \end{bmatrix} = \frac{(p+r)!}{r!},
\end{equation}
which holds for every $r =1,2,3,\ldots\,\,  $. Interestingly enough, for the special case in which $r =p$, identity \eqref{id3}
leads right away to an explicit formula for the Catalan numbers that seems not to have been noticed hitherto (see \eqref{catalan} below).

\section{Derivation of the recursive formula}

To prove formula \eqref{myeq}, we proceed in two steps. In the first step, we state the identity:
\begin{itemize}
\item {\scshape Identity 1}
\begin{equation*}
\sum_{k=1}^{n} \binom{k+p}{p}H_k = \binom{n+p+1}{p+1}H_n - \frac{1}{(p+1)!}
\sum_{t=0}^{p} \begin{bmatrix} p+1 \\ t+1 \end{bmatrix} S_t(n),
\end{equation*}
\end{itemize}
where $H_k = 1 +\frac{1}{2}+\cdots+\frac{1}{k}$ is the $k$th harmonic number. In the second step, we state the identity:
\begin{itemize}
\item {\scshape Identity 2}
\begin{equation*}
\sum_{k=1}^{n} \binom{k+p}{p}H_k = \binom{n+p+1}{p+1}H_n - \frac{1}{p+1}
\left[ \binom{n+p+1}{p+1} - 1 \right].
\end{equation*}
\end{itemize}
Formula \eqref{myeq} then follows by equating the right-hand sides of Identities 1 and 2.

\subsection{Proof of the Identity 1}

To prove Identity 1, we shall use the following lemma.
\begin{lemma}
For any non-negative integer $p$ and for $t =0,1,\ldots,p+1$, we have
\begin{equation}\label{lem:1}
\sum_{k=t}^{p+1} (-1)^{p+1-k} p^{k-t} \binom{k}{t} \begin{bmatrix} p+1 \\ k \end{bmatrix} =
\begin{bmatrix} p+1 \\ t \end{bmatrix}.
\end{equation}
\end{lemma}
\begin{proof}
Let $[x]_p$ denote the falling factorial $x(x-1)(x-2)\ldots (x-p+1)$. Recall that the numbers $\genfrac{[}{]}{0pt}{}{p}{t}$
can be defined algebraically by the relation \cite[equation (6.14)]{temple}
\begin{equation*}
[x]_{p} =\sum_{k=0}^{p} (-1)^{p-k} \begin{bmatrix} p \\ k \end{bmatrix} x^{k}.
\end{equation*}
Thus, we can evaluate $[x+p]_{p+1}$ as
\begin{align}\label{lemma1:1}
[x+p]_{p+1} & =\sum_{k=0}^{p+1} (-1)^{p+1-k} \begin{bmatrix} p+1 \\ k \end{bmatrix} (x+p)^{k} \notag \\
&  = \sum_{k=0}^{p+1} (-1)^{p+1-k} \begin{bmatrix} p+1 \\ k \end{bmatrix} \sum_{t=0}^{k}
\binom{k}{t} p^{k-t} x^t  \notag \\
& =\sum_{t=0}^{p+1}\sum_{k=t}^{p+1} (-1)^{p+1-k} p^{k-t} \binom{k}{t} \begin{bmatrix} p+1 \\ k
\end{bmatrix} x^t.
\end{align}
Now, denoting by $[x]^{p}$ the rising factorial $x(x+1)(x+2)\ldots (x+p-1)$, it is clear that $[x+p]_{p+1} = [x]^{p+1}$.
Furthermore, $[x]^{p+1}$ can be expressed as \cite[equation (6.16)]{temple}
\begin{equation}\label{lemma1:2}
[x]^{p+1} = \sum_{t=0}^{p+1} \begin{bmatrix} p+1 \\ t \end{bmatrix} x^t .
\end{equation}
Therefore, equating coefficients of $x^t$ on the right-hand sides of \eqref{lemma1:1} and \eqref{lemma1:2}, we end up
with relation \eqref{lem:1}.
\end{proof}

Next, we proceed with the proof of the Identity 1:
\begin{align}\label{id11}
\sum_{k=1}^{n} \binom{k+p}{p}H_k & = \sum_{k=1}^{n}\sum_{j=1}^{k} \binom{k+p}{p} j^{-1} =
\sum_{j=1}^{n}\sum_{k=j}^{n} \binom{k+p}{p} j^{-1}  \notag \\
& = \sum_{j=1}^{n} j^{-1} \sum_{k=1}^{n} \binom{k+p}{p} - \sum_{j=1}^{n} j^{-1}
\sum_{k=1}^{j-1} \binom{k+p}{p} \notag \\
& = \binom{n+p+1}{p+1} H_n - \sum_{j=1}^{n} \binom{j+p}{p+1} j^{-1},
\end{align}
where in the last step the following ``hockey stick'' identity \cite[problem 1.7.9.]{temple} is used twice:
\begin{equation}\label{stick}
\sum_{k=1}^{n} \binom{k+p}{p} = \binom{n+p+1}{p+1} -1.
\end{equation}
Furthermore, we have
\begin{align*}
\sum_{j=1}^{n} \binom{j+p}{p+1} j^{-1} & = \sum_{j=1}^{n} j^{-1} \frac{1}{(p+1)!}
\sum_{k=0}^{p+1} (-1)^{p+1-k} \begin{bmatrix} p+1 \\ k \end{bmatrix} (j+p)^{k}  \\
& = \frac{1}{(p+1)!} \sum_{t=0}^{p+1}\sum_{k=t}^{p+1} (-1)^{p+1-k} p^{k-t} \binom{k}{t}
\begin{bmatrix} p+1 \\ k \end{bmatrix} \sum_{j=1}^{n} j^{t-1}.
\end{align*}
By relation \eqref{lem:1}, this reduces to
\begin{equation*}
\sum_{j=1}^{n} \binom{j+p}{p+1} j^{-1} = \frac{1}{(p+1)!} \sum_{t=0}^{p+1}
\begin{bmatrix} p+1 \\ t \end{bmatrix} S_{t-1}(n).
\end{equation*}
Since $\genfrac{[}{]}{0pt}{}{p+1}{0} =0$, this is in turn equivalent to
\begin{equation}\label{id12}
\sum_{j=1}^{n} \binom{j+p}{p+1} j^{-1} = \frac{1}{(p+1)!} \sum_{t=0}^{p}
\begin{bmatrix} p+1 \\ t+1 \end{bmatrix} S_{t}(n).
\end{equation}
Finally, combining equations \eqref{id11} and \eqref{id12} gives Identity 1.

\subsection{Proof of the Identity 2}

To prove Identity 2, we employ the following version of Abel's lemma on summation by parts (see, e.g., \cite{chorlton}) which states:

\noindent Let $\{u_k \}_{k \geq 1}$ and $\{v_k \}_{k \geq 1}$ be two sequences of real numbers with partial sums
$U_n = \sum_{k =1}^{n} u_k$ and $V_n = \sum_{k =1}^{n} v_k$. Further define $U_0 = V_0 =0$. Then
\begin{equation}\label{abel}
\sum_{k=1}^{n} u_k V_k + \sum_{k=1}^{n} v_k U_{k-1} = U_n V_n .
\end{equation}

Hence, letting
\begin{equation*}
\begin{aligned}
v_k &= \frac{1}{k},  &  V_k & = \sum_{j=1}^{k} \frac{1}{j} = H_k,  \\
u_k & = \binom{k+p}{p}, &  U_k & = \sum_{j=1}^{k} \binom{j+p}{p} = \binom{k+p+1}{p+1} -1,
\end{aligned}
\end{equation*}
(where we use \eqref{stick} to get the closed form of $U_k$), and plugging into \eqref{abel}, we obtain
\begin{equation}\label{lemma2:1}
\sum_{k=1}^{n} \binom{k+p}{p}H_k = \binom{n+p+1}{p+1} H_n -
\sum_{k=1}^{n} \frac{1}{k}\binom{k+p}{p+1}.
\end{equation}
Furthermore, noting that $\binom{n}{k+1} = \frac{n-k}{k+1}\binom{n}{k}$, and using \eqref{stick} again,
it follows that
\begin{equation}\label{lemma2:2}
\sum_{k=1}^{n}\frac{1}{k}\binom{k+p}{p+1} = \frac{1}{p+1} \sum_{k=1}^{n} \binom{k+p}{p}
= \frac{1}{p+1} \left[ \binom{n+p+1}{p+1} - 1 \right],
\end{equation}
and thus, from equations \eqref{lemma2:1} and \eqref{lemma2:2}, we obtain Identity 2.

\section{Catalan numbers enter the scene}

First we note that formula \eqref{myeq} can be written in the equivalent form
\begin{equation}\label{cat1}
\sum_{t=0}^{p} r^t  \begin{bmatrix} p+1 \\ t+1 \end{bmatrix} =p! \left[ \binom{r+p+1}{p+1}
- 1 \right] - \sum_{j=1}^{r-1}\sum_{t=0}^{p} j^t \begin{bmatrix} p+1 \\ t+1 \end{bmatrix},
\end{equation}
where $r$ stands for any arbitrary fixed positive integer. In particular, for $r =1$ we retrieve the well-known
relation for the Stirling cycle numbers \cite[equation (6.18)]{temple}
\begin{equation}\label{bcase}
\sum_{t=0}^{p} \begin{bmatrix} p+1 \\ t+1 \end{bmatrix} = (p+1)!,
\end{equation}
which constitutes the base case of the above identity \eqref{id3}. (Note, incidentally, that relation \eqref{bcase}
also follows directly by letting $x=1$ in equation \eqref{lemma1:2}, since $[x]^{p+1} = (p+1)!$ for $x=1$.)
Let us assume as a strong inductive hypothesis that
\begin{equation}\label{cat2}
\sum_{t=0}^{p} j^t \begin{bmatrix} p+1 \\ t+1 \end{bmatrix} = \frac{(p+j)!}{j!},
\end{equation}
for $j=1,2,\ldots, r-1$. Thus, substituting \eqref{cat2} into \eqref{cat1}, we have
\begin{align*}
\sum_{t=0}^{p} r^t  \begin{bmatrix} p+1 \\ t+1 \end{bmatrix} & = p! \left[ \binom{r+p+1}{p+1}
- 1 \right] - \sum_{j=1}^{r-1} \frac{(p+j)!}{j!}  \\
& = p! \left[ \binom{r+p+1}{p+1} - 1 \right] - p! \sum_{j=1}^{r-1} \binom{p+j}{j} .
\end{align*}
Since $\binom{p+j}{j} = \binom{j+p}{p}$, we invoke \eqref{stick} once more to get
\begin{equation*}
\sum_{t=0}^{p} r^t  \begin{bmatrix} p+1 \\ t+1 \end{bmatrix}= p! \left[ \binom{r+p+1}{p+1}
- \binom{r+p}{p+1} \right] = \frac{(p+r)!}{r!}.
\end{equation*}
This completes the inductive step and the proof of the identity \eqref{id3}.

Observe that, for $r=p$, the said identity becomes
\begin{equation*}
\sum_{t=0}^{p} p^t \begin{bmatrix} p+1 \\ t+1 \end{bmatrix} = (p+1)! C_p,
\end{equation*}
where $C_p = \frac{1}{p+1}\binom{2p}{p}$ is the $p$th Catalan number \cite{koshy}. Expressing $C_p$ as
\begin{equation}\label{catalan}
C_p = \frac{\displaystyle \sum_{t=1}^{p+1} p^{t-1} \begin{bmatrix} p+1 \\ t \end{bmatrix}}
{\displaystyle \sum_{t=1}^{p+1} \begin{bmatrix} p+1 \\ t \end{bmatrix}},  \quad\,\,  p \geq 1,
\end{equation}
we can therefore interpret $C_p$ as the average of the function $p^{t-1}$ over all \mbox{$(p+1)!$} permutations
of $p+1$ elements, with $t$ being the number of cycles of a permutation, and $\genfrac{[}{]}{0pt}{}{p+1}{t}$
the number of permutations of $p+1$ elements with exactly $t$ cycles. As an illustrative example, let us apply
\eqref{catalan} to calculate $C_6$. By using the entries in the seventh row of Table \ref{tb:1}, from \eqref{catalan}
we readily obtain
\begin{align*}
C_6 & = \frac{1}{7!}\left( 6^0 \!\begin{bmatrix} 7 \\ 1 \end{bmatrix} + 6^1 \!\begin{bmatrix} 7 \\ 2 \end{bmatrix}
+ 6^2 \!\begin{bmatrix} 7 \\ 3 \end{bmatrix} + 6^3 \!\begin{bmatrix} 7 \\ 4 \end{bmatrix}
+ 6^4 \!\begin{bmatrix} 7 \\ 5 \end{bmatrix} + 6^5 \!\begin{bmatrix} 7 \\ 6 \end{bmatrix}
+ 6^6 \!\begin{bmatrix} 7 \\ 7 \end{bmatrix} \right) \\
& = \frac{1}{5040} \left( 6^0 \cdot 720 + 6^1 \cdot 1764 + 6^2 \cdot 1624  \right. \\[-.2cm]
& \left. \qquad\qquad\,\,\, + \,\, 6^3 \cdot 735  + 6^4 \cdot 175 +6^5 \cdot  21 + 6^6  \cdot 1 \right)  =
\frac{665280}{5040} = 132.
\end{align*}

\vspace{.5cm}

{\it Acknowledgments:} The author is grateful to Thomas Koshy for his encouragement, and the anonymous reviewer for the constructive suggestions for improving an earlier version of the manuscript.

\end{document}